\documentclass{amsart}
\RequirePackage{enumerate, amsmath,amsthm,amssymb, proof, mathrsfs, tikz-cd, wasysym, stmaryrd, verbatim, changepage, parskip, mathtools, csquotes, hyperref}
\RequirePackage[english]{babel}

\let\emph\relax % there's no \RedeclareTextFontCommand
\DeclareTextFontCommand{\emph}{\bfseries\em}

\newtheorem{thm}{Theorem}[section]

\newtheorem{lem}{Lemma}[thm]
\newtheorem{prop}[thm]{Proposition}
\newtheorem{cor}{Corollary}[thm]
\newtheorem{defn}[thm]{Definition}
\newtheorem{conj}[thm]{Conjecture}

\theoremstyle{remark}
\newtheorem{remark}[thm]{Remark}
\newtheorem{ex}[thm]{Example}

\newcommand{\Z}{\mathbb{Z}}
\newcommand{\Q}{\mathbb{Q}}
\newcommand{\R}{\mathbb{R}}
\newcommand{\Zp}{\mathbb{Z}_{(p)}}
\newcommand{\N}{\mathbb{N}}
\newcommand{\Fp}{\mathbb{F}_p}
\newcommand{\m}{\mathfrak{m}}
\newcommand{\n}{\mathfrak{n}}
\newcommand{\M}{\mathcal{M}}
\newcommand{\Span}{\text{Span}}

\newcommand{\Hom}{\text{H}}

\newcommand{\Ass}{\text{Ass}}
\newcommand{\qL}{\mathcal{L}}
\newcommand{\h}{\mathfrak{h}}
\newcommand{\hb}{\underline{\mathfrak{h}}}

\newcommand{\xs}{\underline{x}}
\newcommand{\ys}{\underline{y}}
\newcommand{\ts}{\underline{t}}
\newcommand{\as}{\underline{a}}
\newcommand{\bs}{\underline{b}}
\newcommand{\vs}{\underline{v}}
\newcommand{\ws}{\underline{w}}
\newcommand{\us}{\underline{u}}
\newcommand{\enaught}{\underline{e_0}}

\begin{document}

\title{Content and Q-sequences in Mixed Characteristic Local Rings}

\author{Olivia Strahan}
\address{Department of Mathematics, University of Michigan, East Hall, 530 Church Street, Ann Arbor, 48109, Michigan, United States}
\thanks{This research was partially supported by NSF grants DMS 1840234, DMS 2200501, and DMS~2101075.}
\begin{abstract}
    It is proved that a system of parameters is always a Q-sequence (in the sense of Hochster and Huneke) for several classes of mixed characteristic rings: rings in which the characteristic of the residue field is a nilpotent element, and mixed characteristic analogues of Stanley-Reisner rings, semigroup rings, and toric face rings.
\end{abstract}
\maketitle

\section{Introduction}
\label{sec:intro}

Fix a commutative ring $R$ with unity and a finitely generated ideal $I \subset R$. The notion of \emph{\boldmath $I$-quasilength} is defined for finitely generated modules which are annhiliated by a power of $I$; when the ideal $I$ is maximal, this coincides with the usual notion of length. The \emph{content} is a real number associated to a finitely generated module and a set of generators for $I$, and it is defined as a normalized limit of $I$-quasilengths (see Definition \ref{content}). When the module is $R$ itself, the content is between $0$ and $1$ inclusive, and in fact conjectured to be always equal to either $0$ or $1$. When the content of a ring with respect to a sequence is equal to 1, that sequence is referred to as a \emph{Q-sequence}. Mel Hochster and Craig Huneke introduced these notions in \cite{contentHH}, referring to the content as a heuristic measure of the local cohomology module $H^d_I(M)$.

Of particular interest is the case of the content of a local ring with respect to a system of parameters. Hochster and Huneke showed that the content of an equicharacteristic Noetherian local ring with respect to a system of parameters is always equal to $1$ \cite[Thm 4.7]{contentHH}, and speculated that the content with respect to a system of parameters is $1$ for any local ring; in other words, every system of parameters should be a Q-sequence. They were motivated by connections to the direct summand conjecture (now a theorem, due to Yves Andr\'e \cite{MR3814651}), and believed that the notion of content should be applicable to the question of existence of big Cohen-Macaulay algebras over local rings of either mixed or equal characteristic. This paper will expand Hochster and Huneke's Theorem 4.7 \cite{contentHH} to several classes of mixed characteristic rings. I begin by handling the case of local rings of residue characteristic $p$ where $p$ is a nilpotent element of the ring (Theorem \ref{pnil}). Next, I define mixed characteristic variants of several combinatorial classes of rings: Stanley-Reisner rings, semigroup rings, and toric face rings. The definitions arise by replacing monomials with “$p$-monomials" (see Definition \ref{pmonoms}). Surprisingly, these $p$-monomial constructions preserve many of the usual algebraic properties from the classical setting. I will leverage this to show that systems of parameters are Q-sequences in all three of these variant combinatorial rings. See Theorems \ref{sr}, \ref{bigsec5thm}, and \ref{ptoric}, respectively.

\section{Background}
\begin{defn}[Quasilength {\cite[p 3173]{contentHH}}]
    Let $I$ be a finitely generated ideal of a commutative ring $R$, and let $M$ be a finitely generated $R$-module which is annihilated by some power of $I$. The \emph{\boldmath $I$-quasilength of M}, denoted $\qL_I(M)$, is the minimum length $t$ of a filtration $0 = M_0 \subset M_1 \subset \hdots \subset M_t = M$ such that the factors $M_i / M_{i-1}$ are cyclic and annihilated by $I$.
\end{defn}

Before continuing, we will introduce some notation. Let $\xs = x_1, \hdots, x_n$ be generators for $I$, and let $\ts = (t_1, \hdots, t_d) \in \N^d$. Then define
\[
    I_{\ts} := (x_1^{t_1}, \hdots, x_d^{t_d})R
\]
to be the ideal of $R$ with generators $x_1^{t_1}, \hdots, x_d^{t_d}$.
The choice of generators for $I$ is not directly referenced in the notation, but rather inferred from context.

\begin{defn}[Content {\cite[p 3178]{contentHH}}]
\label{content}
     Let $I = (x_1, \hdots, x_d) \subset R$ be an ideal, and let $M$ be a finitely generated $R$-module. The \emph{content} is defined as
    \[
        \h_{\xs}^d(M) := \lim_{s \to \infty} \inf \left\{  \frac{\qL_I(M/I_{\ts}M)}{t_1 \hdots t_d}  : \ts \in \N^d \; t_i \geq s \; \forall i \right\}
    \]
\end{defn}

The content of a module $M$ is, a priori, a nonnegative real number bounded above by the minimum number of generators for $M$ as an $R$-module \cite[Prop 2.1]{contentHH}. In particular, for any sequence of elements $x_1, \hdots, x_d \in R$, one has that $0 \leq \h_{\xs}^d(R) \leq 1$, even if the ring $R$ is non-Noetherian. However, we have no examples of a case where the inequalities are both strict. This leads to the following conjecture:

\begin{conj}[Dichotomy]
    \label{dichotomy}
    Let $R$ be any commutative ring, and let $x_1, \hdots, x_d$ be any finite set of elements of $R$. Then the content $\h_{\xs}^d(R)$ is equal to either $0$ or $1$.
\end{conj}

Hochster and Huneke prove this to be true in prime characteristic rings \cite[Thm 3.9]{contentHH}. Loosely speaking, the content being $1$ should be thought of as a non-degeneracy condition on the sequence $\xs$. Throughout this paper I will use the following terminology introduced by Hochster and Huneke:

\begin{comment}
{\color{red} TODO: check up on if Thm 3.9 can be extended to equichar 0 by reduction to char p {\color{blue} nope they say in the abstract that dichotomy is still an open question in equichar 0 and mixed char, and i presume they would have tried reduction to char p}. Also check if Thm 3.9 requires Noetherian {\color{blue} not as far as I can tell}}
    
\end{comment}

\begin{defn}[Q-sequence {\cite[page 3184]{contentHH}}]
    Let $R$ be any commutative ring. A finite set of elements $x_1, \hdots, x_d$ is called a \emph{Q-sequence} if $\h_{\xs}^d(R) = 1$.
\end{defn}

\begin{remark}
    \label{hbcontent2}
    Hochster and Huneke in fact define two distinct notions of content. The other variant, denoted $\hb_{\xs}^d(M)$, is somewhat more complicated to define (refer to \cite[p 3177]{contentHH}). Fortunately, the definition of a Q\babelhyphen{nobreak}sequence does not depend on which version of content is used! While it is not known whether or not $\hb_{\xs}^d(R)$ is always equal to $\h_{\xs}^d(R)$, it does turn out to be the case that $\hb_{\xs}^d(R) = 1$ if and only if $\h_{\xs}^d(R) = 1$ \cite[Thm 3.8]{contentHH}. The main results of this paper are about Q-sequences; therefore I need only use the definition of content that is easier to state.
\end{remark}

\begin{thm}[{\cite[Thm 4.7]{contentHH}}]
    \label{contentHH4.7}
    Let $R$ be an equicharacteristic Noetherian local ring. Then every system of parameters for $R$ is a Q-sequence.
\end{thm}

\begin{cor}[{\cite[Cor 4.9]{contentHH}}]
    \label{contentHH4.9}
    Let $R$ be any commutative ring, and let $x_1, \hdots, x_d$ be a finite set of elements of $R$. If there exists a Noetherian $R$-algebra $S$ such that $(\xs)S$ has height $d$, then the content $\h_{\xs}^d(R)$ is positive. If $R$ contains a field then in fact $\h_{\xs}^d(R) = 1$.
\end{cor}
\begin{comment}
    {
\color{red} If Noetherian isn't needed {\color{blue} (I'm pretty sure that it's not)}, can I think of any interesting example of a non-Noetherian ring with a Noetherian algebra
   {\color{blue} If $R = A[X_1, X_2, \hdots]$ and $S = A[X_1] = R/(X_2, X_3, \hdots)$ then $S$ Noetherian $R$ algebra and height$(X_1S) = 1$, although it already had height 1 in $R$ so idk if that is very interesting. Ask Karen if for a general map $R \to S$ there is any relationship between ht(I) and ht(IS)} {\color{red} Nope. The height can go down even if the map is injective  (eg, $R \hookrightarrow \text{frac}(R)$}), and the height can also go up even if the map is finite type (eg, look at the image of $(y,z)$ under the map $K[x,y,z]/(xy,xz) \to K[y,z] \; \; x \mapsto 0$)

Mel/Craig probably included non-Noetherian bc of $R^+$ (can restate direct summand conjecture in terms of $R^+$)
}
\end{comment}

\begin{comment}
    Let $R$ be any Noetherian commutative ring, and let $\xs = x_1, \hdots, x_d \in R$. Then $(\xs)R$ has a minimal prime of height $d$ if and only if $\xs$ is a system of parameters for a local ring of $R$. {\color{red} Come back! this might not actually be quite the correct statement}
\end{comment}

\begin{remark}
    The hypotheses of Corollary \ref{contentHH4.9} are equivalent to the condition that $\xs$ maps to a system of parameters in some local Noetherian $R$-algebra (which is equicharacteristic when $R$ contains a field). If $R$ itself is Noetherian, then in particular these conditions are satisfied if the ideal $(\xs)R$ has any minimal prime of height $d$.
\end{remark}

\begin{conj}
\label{sopcont}
Let $R$ be a (Noetherian) local ring of dimension $d$. Then every system of parameters for $R$ is a Q-sequence.

More generally, if $R$ is a (Noetherian) ring, and $\xs = x_1, \hdots, x_d$ is a sequence such that $(\xs)R$ has a minimal prime of height $d$, then $\xs$ is a Q-sequence.
\end{conj}

The hope is that the equicharacteristic condition in Theorem \ref{contentHH4.7} can in fact be dropped. Perhaps even the Noetherian condition can be dropped, but that aspect of the problem is not examined in this paper. One good reason to be hopeful is that, by the positivity result in Corollary \ref{contentHH4.9}, the Noetherian case of Conjecture \ref{sopcont} would follow from Conjecture \ref{dichotomy} (Dichotomy). The goal of this paper is to affirm Conjecture \ref{sopcont} for several classes of mixed characteristic rings. The following propositions will be used frequently:

\begin{prop}
    \label{alg}
    Let $R \to S$ be a map of commutative rings, and let $x_1, \hdots, x_d$ be elements of $R$ with images $y_1, \hdots, y_d$ in $S$. If $\ys$ is a Q-sequence in $S$, then $\xs$ is a Q-sequence in $R$.
\end{prop}
\begin{proof}
    This follows directly from Proposition 2.5 \cite{contentHH}, stated below:
    \begin{prop}[{\cite[Prop 2.5]{contentHH}}]
       \label{contentHH2.5}
       If $R \to S$ is a map of commutative rings, $M$ is a finitely generated $R$-module, and $x_1, \hdots, x_d \in R$ with images $y_1, \hdots, y_d \in S$, then $\h_{\xs}^d(M) \geq \h_{\ys}^d(S \otimes M)$. In particular, $\h_{\xs}^d(R) \geq \h_{\ys}^d(S)$.
    \end{prop}
\end{proof}

\begin{remark}
    It follows from Proposition \ref{alg} that the non-local statement of Conjecture \ref{sopcont} is equivalent to the local version.
    %There are many ways to use Proposition \ref{alg} to get consequences of Conjecture \ref{sopcont} for non-local rings. For example, Conjecture \ref{sopcont} would imply that a sequence of length $d$ which generates an ideal of height $d$ is always a Q-sequence.
\end{remark}

\begin{prop}
    \label{minprimeQseq} Let $R$ be a commutative ring, and suppose that $R/P$ satisfies Conjecture \ref{sopcont} for every minimal prime $P$ of $R$; in other words, any length $d$ sequence in $R/P$ generating an ideal of height $d$ is a $Q$-sequence in $R/P$. Then $R$ also satisfies Conjecture \ref{sopcont}.
\end{prop}
\begin{proof}    
    Let $\xs = x_1, \hdots, x_d$ be elements of $R$ such that $(\xs)R$ has a minimal prime $Q$ of height $d$. I will show that $\xs$ is a Q-sequence in $R$.
    
    There is a minimal prime $P \subset Q$ of $R$ such that $Q/P$ is a height $d$ prime of $R/P$. Let $\pi$ be the quotient map $R \to R/P$. Then $Q/P$ is a height $d$ minimal prime of $(\pi(\xs))R/P$. By assumption, it follows that $\pi(\xs)$ is a Q-sequence on $R/P$, and thus $\xs$ is a Q-sequence on $R$ by Proposition \ref{alg}.
\end{proof}

\begin{prop}
    \label{reg}
    Let $R$ be a commutative ring with a sequence of elements $\xs = x_1, \hdots, x_d$. If $\xs$ is a regular sequence on some Noetherian $R$-module, then $\xs$ is a Q-sequence. It follows that if $R$ has a Noetherian Cohen-Macaulay module, any length $d$ sequence generating an ideal of height $d$ is a Q-sequence.
\end{prop}
\begin{proof}
    If $\xs$ is a regular sequence on some Noetherian $R$-module, then by \cite[Prop 1.2(c)]{contentHH}, we have $\qL_I(R/I_{\ts}) = t_1 \hdots t_d$ for every $\ts$, and thus $\h_{\xs}^d(R) = 1$.

    Now suppose that $M$ is a Noetherian Cohen-Macaulay $R$-module and $x_1, \hdots, x_d$ generates an ideal of height $d$. Let $P \subset R$ be a height $d$ minimal prime of $(\xs)R$. The images of $\xs$ form a system of parameters for $R_P$, and $M_P$ is a Noetherian Cohen-Macaulay $R_P$-module. Then $\xs$ is a regular sequence on $M_P$, so $\h_{\xs}^d(R_P) = 1$ as above. We have $1 \geq \h_{\xs}^d(R) \geq \h_{\xs}^d(R_P) = 1$ by Proposition \ref{contentHH2.5}, so in fact $\h_{\xs}^d(R) = 1$.
\end{proof}

\begin{remark}
    Propositions \ref{alg} and \ref{reg} work together very well. To show that $\xs$ is a Q-sequence on $R$, one may look for a module $M$ such that $\xs$ are a regular sequence on $M$, where $M$ need only be Noetherian as a module over any $R$-algebra. This is very convenient because it allows us to localize with impunity. The ring $R_P$ is Noetherian when $R$ is Noetherian, but very rarely is it Noetherian as an $R$-module.
\end{remark}

\section{Nilpotent prime}
In this section, we prove Conjecture \ref{sopcont} for local rings whose residue characteristic is nilpotent. 

\begin{thm}
    \label{pnil}
    Let $(R, \m)$ be a Noetherian local ring of dimension $d$ and residue characteristic $p > 0$ such that $p$ is nilpotent in $R$. Then every system of parameters for $R$ is a Q-sequence.
\end{thm}

\begin{remark}
    \label{Zpalg}
    When the residue characteristic of a local ring is $0$, then the ring contains $\Q$ and is therefore equal characteristic. Otherwise, if $(R, \m)$ is a local ring with residue characteristic $p$, then there is a natural map of local rings $\Zp \to R$, and the kernel of this map is either $0$ or $p^t \; \Zp$ for some $t\in \N$. Theorem \ref{pnil} covers precisely the case when this map is \textit{not} injective.
\end{remark}

\begin{proof}[Proof of Theorem \ref{pnil}]
    By Proposition \ref{minprimeQseq}, it is enough to check that systems of parameters are Q-sequences in $R/P$ for every minimal prime $P$. So, let $P$ be any minimal prime of $R$. Since $p \in R$ is nilpotent, one has that $p\in P$ and  $\Z/p^t\Z \subseteq R$ is a subring for some $t > 0$. Moreover, every element of $\Z/p^t\Z$ is either a unit (so not in $P \subseteq \m$) or divisible by $p$. Therefore
    \[
      P \cap \Z/p^t\Z = (p) \Z/p^t\Z \quad \text{and} \quad \text{im}(\Z/p^t\Z \to R \to R/P) = \Fp \subseteq R/P,
    \]
    so $R/P$ is a Noetherian local ring of equal characteristic $p$. By Theorem \ref{contentHH4.7}, it follows that every system of parameters for $R/P$ is a Q\babelhyphen{nobreak}sequence.
\begin{comment}
    Since $p \in R$ is nilpotent, one has that $\Z/p^t\Z \subseteq R$ is a subring for some $t > 0$ and $p$ is contained in every minimal prime of $R$. Choose a minimal prime $\mathfrak{p}$ of $R$ such that $d = \dim(R/\mathfrak{p})$. Then the image of $\xs$ is still a system of parameters for $R/\mathfrak{p}$. Moreover, every element of $\Z/p^t\Z$ is either a unit (so not in $\mathfrak{p} \subseteq \m$) or divisible by $p$. Therefore
\[
      \mathfrak{p} \cap \Z/p^t\Z = (p) \Z/p^t\Z \quad \text{and} \quad \text{im}(\Z/p^t\Z \to R \to R/\mathfrak{p}) = \Fp \subseteq R/\mathfrak{p}.
\]
This means that $R/\mathfrak{p}$ is a Noetherian local ring of equal characteristic $p$, for which $\xs$ is a system of parameters. It follows from Theorem \ref{contentHH4.7} that $\xs$ is a Q-sequence on $R/\mathfrak{p}$, and then from Prop \ref{alg} that $\xs$ is a Q-sequence on $R$.
\end{comment} 
\end{proof}

\begin{cor}
Let $R$ be Noetherian with positive ring characteristic, and suppose $x_1, \hdots, x_d$ generate a height $d$ ideal in $R$. Then $\xs$ is a Q-sequence.
\end{cor}
\begin{proof}
    Let $P$ be a height $d$ minimal prime of $(\xs)R$, so that $\xs$ is a system of parameters for $R_P$. Since $R$ has a positive ring characteristic, the natural~map
    \[
         \Z \to R \to R_P \to R_P / PR_P
    \]
    is not injective. Therefore $R_P / PR_P$ is a field with positive characteristic $p$, where $p \in R_P$ is nilpotent.
    It follows from Theorm \ref{pnil} that $\xs$ is a Q-sequence on $R_P$, and from Proposition \ref{alg} that $\xs$ is a Q-sequence on $R$.
\end{proof}

\section{\texorpdfstring{$p$-monomials}{p-monomials}}
\label{pmonomsec}
We will now introduce a scheme for constructing new classes of mixed characteristic rings using “$p$-monomials," then discuss the basic properties of these objects and their relationship to the usual type of monomials.

Throughout this section, $S$ shall be a $\Z^n$-graded $\Zp$-algebra with $n$ variables. The primary case is the Laurent polynomial ring $S = \Zp[x_1^{\pm 1}, \hdots, x_n^{\pm 1}]$, but some variations include the polynomial ring or even the power series ring (with some care, since the power series ring is not graded). One could also expand the coefficients to any $\Zp$-algebra with the property that all elements can be uniquely written as a unit times a power of $p$; note that $\Q$ is such a $\Zp$ algebra, for example. The coefficient ring $(\Zp, p)$ could also be replaced with any mixed characteristic DVR $(V,t)$, in which case we would work with “$t$-monomials" instead of $p$-monomials.

% Throughout this section, $S$ shall be a $\Z^n$-graded $\Zp$-algebra with $n$ variables. Some examples include the power series ring $\Zp [[x_1, \hdots, x_n]]$, the polynomial ring $\Zp[x_1, \hdots, x_n]$, and the Laurent polynomial ring $\Zp[x_1^{\pm 1}, \hdots, x_n^{\pm 1}]$. The coefficient ring can also be expanded to the $p$-adics or any other mixed-characteristic DVR $V$ with uniformizer $p$. In all these settings, the ring $S$ has an $\Z^n$-grading given by multidegree in the $n$ variables; {\color{red} wait a minute, the power series ring isn't actually graded! (it's not a direct sum of its homog components)}

%We can expand the setting even further if we like, by allowing the coefficient ring to be any mixed characteristic DVR $(V,t)$ where $t$ is not necessarily the image of a prime integer (such as, for example, $\Zp[[t]]/(t^2 - p)$). In this case, we would need to work with $t$-monomials instead of $p$-monomials.

\begin{defn}[$p$-monomials]
    \label{pmonoms}
    A \emph{\boldmath $p$-monomial} is an element of $S$ of the form $p^{t_0} x_1^{t_1} \hdots x_n^{t_n}$, where $\ts = (t_0, t_1, \hdots, t_n) \in \Z^{n+1}$.
\end{defn}
\begin{remark}
    \label{homog}
    Up to multiplication by a unit, these are the only $\Z^n$-homogeneous elements of $S$.
\end{remark}

\begin{remark}
      I will adapt standard terminology for monomials to the setting of $p$-monomials in the obvious way; for example, a \emph{$p$-monomial ideal} shall refer to an ideal generated by $p$-monomials, et cetera. It turns out that $p$-monomials share many good computational properties with monomials. This is somewhat surprising, considering that $p$-monomials and monomials have very different additive properties \textemdash namely, it is possible to add $p$-monomials together and get a completely different $p$-monomial! If $w$ is a $p$-monomial, then $w + \hdots + w = pw$ is also a $p$-monomial. Consequently there are some tools developed for monomials, such as Gr\"obner bases, which cannot be directly lifted into the $p$-monomial setting.
\end{remark}

% \begin{prop}
%     \label{pmonominclusion}
%     Let $w_1, \hdots, w_m \in S$ be $p$-monomials.
%     \begin{enumerate}
%         \item Let $I = (w_1, \hdots, w_m)S$ be a $p$-monomial ideal. For any $p$-monomial $w$,
%         \[
%             w \in I \text{ if and only if } w \in (w_i)S \text{ for some } i.
%         \]

%         \item Let $R = \Zp[w_1, \hdots, w_m] \subset S$ subring. For any $p$-monomial $w$,
%         \[
%            w \in R \text{ if and only if } w = p^{h_0}w_1^{h_1}\hdots w_m^{h_m} \text{ for some } \vec{h} \in \N^{m+1}.
%         \]
%     \end{enumerate}
% \end{prop}
% \begin{proof}
%     I will prove only the first statement; the proof of the second statement is similar in spirit and heavier in notation.

%     It is clear that $w \in (w_i)S$ implies $w \in I$. Suppose that $w = p^{t_0}\xs^{\ts} \in I$, and write $p^{t_0}\xs^{\ts} = w_1 f_1 + \hdots + w_m f_m$ for some $f_i \in S$. By projecting to the graded component of multidegree $\ts$, we may assume that every term on the right hand side is of the form $w_i f_i = c_i p^{e_i} \xs^{\ts}$, where $c_i \in \Zp$ is a unit. Let $e$ be the minimum of the $e_i$. Clearing the monomial $\xs^{\ts}$ from both sides of the equation, we are left with
%     \[
%         p^{t_0} = c_1 p^{e_1} + \hdots + c_m p^{e_m} \in p^e \Zp.
%     \]
%     This can only happen if $t_0$ is greater than or equal to $e$, in which case there is some $i$ for which $t_0 \geq i$, and therefore $w = p^{t_0}\xs^{\ts} = p^{t_0 - e_i} \cdot c_i^{-1}w_i f_i \in (w_i)S$.
% \end{proof}

\begin{thm}
    \label{phi}
    Let $S'$ be the $S$-algebra with one additional variable $x_0$, and let $\phi:S' \to S$ be the evaluation map $x_0 \mapsto p$.
    \begin{enumerate}
        \item The map $\phi$ is surjective and induces a bijective correspondence between monomials in $S'$ and $p$-monomials in $S$.
        \item If $I \subset S$ is a $p$-monomial ideal, then there is a unique monomial ideal $J \subset S'$ such that $p^{t_0}\xs^{\ts} \in I \text{ if and only if } x_0^{t_0}\xs^{\ts} \in J$. Moreover, $I$ is prime if and only if $J$ is prime, and this correspondence commutes with radicals in the sense that $\phi(\sqrt{J})S = \sqrt(I)$.
        \item If $R \subset S$ is a subalgebra generated by $p$-monomials, then there is a unique subalgebra $R' \subset S'$ generated by monomials such that $p^{t_0}\xs^{\ts} \in R \text{ if and only if } x_0^{t_0}\xs^{\ts} \in R'$. The evaluation map $\phi$ restricts to a surjective map $R' \to R$, where the kernel is the principal ideal $(x_0 - p)R'$.
    \end{enumerate}
\end{thm}
\begin{proof}
    %(1) is clear. (2) and the first part of (3) follow almost immediately from Proposition \ref{pmonominclusion}. It remains only to show that the kernel of $\phi$ restricted to $R' \to R$ is generated by $x_0 - p$. 
    (1) is clear. For parts (2) and (3), I will use the following lemma:
    \begin{lem}
        \label{pmonominclusion}
        Let $w_1, \hdots, w_m \in S$ be $p$-monomials.
        \begin{enumerate}
            \item Let $I = (w_1, \hdots, w_m)S$ be a $p$-monomial ideal. For any $p$-monomial $w$,
            \[
                w \in I \text{ if and only if } w \in (w_i)S \text{ for some } i.
            \]
            \item Let $R = \Zp[w_1, \hdots, w_m] \subset S$ subring. For any $p$-monomial $w$,
            \[
                w \in R \text{ if and only if } w = p^{h_0}w_1^{h_1}\hdots w_m^{h_m} \text{ for some } \vec{h} \in \N^{m+1}.
            \]
    \end{enumerate}
    \end{lem}
    \begin{proof}
        I will prove only the first statement; the proof of the second statement is similar in spirit and heavier in notation.
        
        It is clear that $w \in (w_i)S$ implies $w \in I$. Suppose that $w = p^{t_0}\xs^{\ts} \in I$, and write $p^{t_0}\xs^{\ts} = w_1 f_1 + \hdots + w_m f_m$ for some $f_i \in S$. By projecting to the graded component of multidegree $\ts$, we may assume that every term on the right hand side is of the form $w_i f_i = c_i p^{e_i} \xs^{\ts}$, where $c_i \in \Zp$ is a unit. Let $e$ be the minimum of the $e_i$. Clearing the monomial $\xs^{\ts}$ from both sides of the equation, we are left with
        \[
            p^{t_0} = c_1 p^{e_1} + \hdots + c_m p^{e_m} \in p^e \Zp.
        \]
        This can only happen if $t_0$ is greater than or equal to $e$, in which case there is some $i$ for which $t_0 \geq i$, and therefore $w = p^{t_0}\xs^{\ts} = p^{t_0 - e_i} \cdot c_i^{-1}w_i f_i \in (w_i)S$.
    \end{proof}

    It is not hard to check that (2) and the first part of (3) follow from Lemma \ref{pmonominclusion}. It remains only to show that the kernel of $\phi$ restricted to $R'$ is $(x_0-p)R'$.

    The kernel of $\phi|_{R'}$ is $(x_0 - p)S' \cap R'$, which certainly contains $(x_0 - p)R'$. Suppose on the other hand that $f = (x_0 - p)g$ is an element of $(x_0 - p)S' \cap R'$. I will show that in fact $g \in R'$.

    Write $\Z^{n+1}$-homogeneous decompositions
    \[
                 f = \sum_{t_0, \ts \in \Z^{n+1}} c(t_0, \ts) \cdot x_0^{t_0}\xs^{\ts}
                 \qquad \text{and} \qquad 
                 g = \sum_{t_0, \ts \in \Z^{n+1}} b(t_0, \ts) \cdot x_0^{t_0}\xs^{\ts}
    \]
    for $f$ and $g$. Note that these are in fact finite sums. So, let $m \in \Z$ be minimal such that some monomial $x_0^m \xs^{\ts}$ appears with nonvanishing coefficient in $g$. With this notation, we have:
    \begin{align*}
        f &= (x_0 - p)g\\
        &= \left( \sum  b(t_0, \ts) \cdot x_0^{t_0+1}\xs^{\ts} \right) + \left( \sum  -p \cdot b(t_0, \ts) \cdot x_0^{t_0}\xs^{\ts} \right)\\
        &= \left( \sum_{t_0 > m}  b(t_0-1, \ts) \cdot x_0^{t_0}\xs^{\ts} \right) + \left( \sum_{t_0 \geq m}  -p \cdot b(t_0, \ts) \cdot x_0^{t_0}\xs^{\ts} \right)
    \end{align*}
    Comparing the coefficient of $x_0^{t_0}\xs^{\ts}$ on both sides, we get that
    \begin{equation} \label{coeffs}
        c(t_0, \ts) = \begin{cases}
            b(t_0 - 1, \ts) - p b(t_0, \ts) & \text{if } t_0 > m\\
            -p b(t_0, \ts) & \text{if } t_0 = m.
        \end{cases}
    \end{equation}
    Since $R'$ is $\Z^{n+1}$-graded and $f \in R'$, it follows that $x_0 \xs^{\ts} \in R'$ whenever $c(t_0, \ts)$ is nonzero. I will use this fact to prove by induction on $t_0$ that $x_0 \xs^{\ts} \in R'$ whenever $b(t_0, \ts)$ is nonzero (and thus $g \in R'$ as desired).

    Let $t_0, \ts \in \Z^{n+1}$ be such that $b(t_0, \ts) \neq 0$. If $t_0 = m$, it follows from equation~\ref{coeffs} that $c(t_0, \ts) = -pb(t_0, \ts) \neq 0$. Therefore $x_0 \xs^{\ts} \in R'$.

    Now let $t_0 > m$, and suppose for contradiction that $x_0 \xs^{\ts} \notin R'$. Then
    \begin{align*}
        &c(t_0, \ts) = 0 = b(t_0-1, \ts) - pb(t_0, \ts) & &\text{by \ref{coeffs}}\\
        &0 \neq p b(t_0, \ts) = b(t_0 - 1, \ts)\\
        &x_0^{t_0-1}\xs^{\ts} \in R' & &\text{by inductive hypothesis}\\
        &x_0^{t_0}\xs^{\ts} = x_0 \cdot x_0^{t_0-1}\xs^{\ts} \in R' & &\text{since } x_0 \in R',
    \end{align*}
    contradicting the assumption that $x_0^{t_0} \xs^{\ts} \notin R'$. This finishes the proof that $g \in R'$ and $f = (x_0 - p)g \in (x_0 - p)R'$.
\end{proof}

\section{Mixed characteristic analogue of Stanley-Reisner rings}
The ring $S$ will continue to be a $\Z^n$-graded $\Zp$-algebra with $n$ variables, as in the previous section. The quotient of $S$ by a monomial ideal  is called a \emph{Stanley-Reisner ring}. I will show that Conjecture \ref{sopcont} holds for a Stanley-Reisner variant, where the ideal is instead generated by $p$-monomials.

\begin{thm}
    \label{sr}
    Let $R = S/I$, where $I$ is a $p$-monomial ideal. Then every length $d$ sequence in $R$ with a height $d$ minimal prime is a Q-sequence on $R$.
\end{thm}

\begin{remark}
    To prove Theorem \ref{sr}, I will classify the minimal primes of $R$ and check that the statement holds modulo each minimal prime (see Proposition \ref{minprimeQseq}). Analogously to the case of monomial ideals, the minimal primes of a $p$-monomial ideal are generated by subsets of $\{p, x_1, \hdots, x_n\}$:
\end{remark}

\begin{lem}
    \label{srminprimes}
    Let $I$ be a $p$-monomial ideal, and let $J$ be the monomial ideal such that $\phi(J)S = I$, where $\phi$ is the evaluation map $\phi(x_0) = p$. Then
    \[
       \text{Min}(I) = \{\phi(Q)S : Q \in \text{Min}(J)\}.
    \]
    Therefore the minimal primes of $I$ are generated by subsets of $\{p, x_1, \hdots, x_n\}$.
\end{lem}
\begin{proof}
    By part (2) of Theorem \ref{phi}, we may assume without loss of generality that $I$ and $J$ are radical. Write $J = Q_1 \cap \hdots \cap Q_h$, where the $Q_i$ are the minimal primes of $J$. Note that each $Q_i$ is generated by a subset of the variables $\{x_0, x_1, \hdots, x_n\}$, and so the corresponding $p$-monomial ideals $P_i := \phi(Q_i)S$ are generated by subsets of $\{p, x_1, \hdots, x_n\}$. In order to show that $I = P_1 \cap \hdots \cap P_h$, it suffices to check for $p$-monomials, since all of these ideals have $p$-monomial generators. Indeed,
    \begin{align*}
        p^{t_0}\xs^{\ts} \in I &\text{ if and only if } x_0^{t_0}\xs^{\ts} \in J\\
        &\text{ if and only if } x_0^{t_0}\xs^{\ts} \in Q_i \text{ for every } i\\
        &\text{ if and only if } p_0^{t_0}\xs^{\ts} \in P_i \text{ for every } i.
    \end{align*}
    I have shown that the minimal primes of $I$ are among $P_1, \hdots, P_h$. Since $I$ is homogeneous, its minimal primes are also homogeneous (and therefore generated by $p$-monomials). Since the correspondence between monomial and $p$-monomial ideals is inclusion-preserving, it follows that $P_1, \hdots, P_h$ are precisely the minimal primes of $I$.
\end{proof}
    
We are now equipped to prove the theorem:
\begin{proof}[Proof of Thm \ref{sr}]
    For any $P \subset R$ minimal, by Lemma \ref{srminprimes}, we have either
    \[
    \frac{R}{P} \cong \frac{S}{(x_{i_1}, \hdots, x_{i_k})} \qquad \text{or} \qquad  \frac{R}{P} \cong \frac{\Fp \otimes S}{(x_{i_1}, \hdots, x_{i_k})}.
    \]
    Either way, the quotient is a Noetherian Cohen-Macaulay ring, and thus satisfies Conjecture \ref{sopcont} by Proposition \ref{reg}. It follows from Proposition \ref{minprimeQseq} that $R$ satisfies Conjecture \ref{sopcont} as well.
\end{proof}

\section{\texorpdfstring{Semigroup and $p$-semigroup rings}{Semigroup and p-semigroup rings}}
We now examine Conjecture \ref{sopcont} for finitely generated $\Z^n$-graded $\Zp$-subalgebras of the Laurent polynomial ring in $n$ variables, a class which I will refer to as \emph{\boldmath affine $p$-semigroup rings}. The goal is to prove the following result:
\begin{thm}
    \label{bigsec5thm}
    Let $R$ be an affine $p$-semigroup ring. Then any length $d$ sequence in $R$ with a minimal prime of height $d$ is a Q-sequence on $R$.
\end{thm}

I will first consider the special case of a subring generated by finitely many monomials $x_1^{t_1}\hdots x_n^{t_n}$, a class of rings which is classically referred to as \emph{affine semigroup rings}. Next I will generalize to  finitely-generated multigraded subalgebras; note that these are subalgebras generated by finitely many $p$-monomials. I refer to this class of rings as \emph{$p$-semigroup rings}. The correspondence from Theorem \ref{phi} will allow us to reduce the general case to the semigroup case. The strategy in both cases is to prove that the normalization of the ring is Noetherian and Cohen-Macaulay (see Proposition \ref{reg}).

I will now state a few results which are central to my arguments. The following two statements are due to a paper of Mel Hochster \cite[Prop 1, Thm~1]{toricH}:

\begin{prop}
    \label{toricHprop1}
      Let $M$ be a semigroup of monomials in the variables $x_1, \hdots, x_n$. Then the following are equivalent:
      \begin{enumerate}[(1)]
          \item The semigroup ring $K[M]$ is normal and Noetherian for some field $K$.
          \item $M$ is a normal semigroup.
          \item  $M$ is finitely generated as a semigroup and for every Noetherian normal domain $D$, the subring $D[M] \subset D[x_1, \hdots, x_n]$ is normal.
      \end{enumerate}
  \end{prop}  

\begin{thm}
    \label{toricHthm1}
    Let $M$ be a normal semigroup of monomials. Then $A[M]$ is Cohen-Macaulay for every Cohen-Macaulay ring $A$.
\end{thm}

Taken together, these imply that a finitely generated normal semigroup ring over a normal Cohen-Macaulay ring of coefficients is Cohen-Macaulay. This is the linchpin of my argument for the semigroup ring case.

It should be noted that Hochster states these results only for the case where the semigroup of monomial exponent vectors is contained in $\N^n$. However, every pointed semigroup is isomorphic to a subsemigroup of $\N^n$ \cite[Cor 7.23]{cca}, so his results can be safely generalized to any pointed semigroup. Sections \ref{sgsubsec} and \ref{psgsubsec} will deal only with pointed semigroups. In Section \ref{locsec}, I will argue that all affine $p$-semigroup rings are localizations of pointed affine $p$-semigroup rings, and use this to finish the proof of Theorem \ref{bigsec5thm}.

For both the semigroup and $p$-semigroup cases, I will also use the following theorem from a textbook of Irena Swanson and Craig Huneke \cite[Thm 2.3.2]{intclosSH}:

\begin{thm}
    \label{intclosSHthm}
    Let $d,e \in \N$, let $G = \N^d \times \Z^e$, and let $R \subseteq S$ be commutative $G$-graded rings. Then the integral closure of $R$ in $S$ is $G$-graded.
\end{thm}

\subsection{Semigroup rings}
\label{sgsubsec}
\begin{defn}[Monomials and semigroup rings]
    Let $A$ be a commutative domain, and let $T = A[x_1, \hdots, x_n, x_1^{-1}, \hdots, x_n^{-1}]$. A \emph{monomial} is an element of the form $\xs^{\ts} = x_1^{t_1} \hdots x_n^{t_n}$ for some exponent vector $\ts = (t_1, \hdots, t_n) \in \Z^n$. 

    Let $M \subset \Z^n$ be a semigroup with identity. The \emph{\boldmath semigroup ring defined by $M$} is the subring $A[M] \subseteq T$ which is generated as an $A$-algebra by all monomials $\xs^{\ts}$ such that the exponent vector $\ts$ is in $M$.
    
    If $M$ is finitely generated, then $A[M]$ is called an \emph{affine semigroup ring}. If the cone $\R_{\geq 0} M$ is pointed, then $A[M]$ is called a \emph{pointed semigroup ring}.
\end{defn}

\begin{remark}
    Throughout this paper, a “semigroup" means a semigroup with an identity element. Such structures are also referred to as \textit{monoids}. The distinction is not terribly important; if $M$ is a semigroup \textit{without} identity, then $M' = M \cup \{0\}$ is a monoid and $A[M] = A[M']$, since $\xs^{0} = 1 \in A$.
\end{remark}

\begin{remark}
    \label{assocsemigp}
    When $R = A[M]$ is a semigroup ring, the set of monomials in $R$ is isomorphic to $M$ as a semigroup, and a set of monomials generates $R$ as an $A$-algebra if and only if their exponent vectors generate $M$ as a semigroup. I will refer to $M$ as the \emph{semigroup associated to} $R$.

    Note that $A[M]$ has a natural $\Z^n$-grading given by multidegree. The semigroup structure of $M$ captures all the information about the multigraded $A$-algebra structure of $A[M]$, in the sense that isomorphisms of semigroups induce multigraded $A$-algebra isomorphisms and vice versa.
\end{remark}

\begin{prop}
    \label{sgmaxhomog}
    If $M \subseteq \Z^n$ is a pointed semigroup and $(A, \n)$ is a local ring, then the semigroup ring $R = A[M]$ has a unique maximal $\Z^n$-homogeneous ideal $\m = \n R + (\xs^{\ts} : \ts \in M - \{0\})R$.
\end{prop}
\begin{proof}
    If $M \subseteq \N^n$, then $R$ is a subring of the polynomial ring $A[\xs]$ and its maximal homogeneous ideal $\m$ is simply the restriction $(\xs)A[\xs] \cap R$. Otherwise, there is an isomorphism $R \to A[M']$ for some semigroup $M' \subset \N^n$, and $\m$ is the preimage of the homogeneous maximal ideal of $A[M']$.
\end{proof}

The goal of this section is to prove Theorem 5.1 for pointed affine semigroup rings, as stated here in Theorem \ref{sgcontent}:

\begin{thm}
    \label{sgcontent}
    Let $R$ be a pointed affine semigroup ring. Then any length $d$ sequence with a minimal prime of height $d$ is a Q-sequence in $R$.
\end{thm}

\begin{remark}
    \label{sgcontentproof}
    By Proposition \ref{reg}, it is enough to find a Noetherian Cohen-Macaulay $R$-module. I claim that the normalization of $R$ is such a module. 
    
    First, I will show that the normalization of $R$ is also a semigroup ring. I will then apply Mel Hochster's theorem on normal semigroup rings to conclude that the normalization is a Noetherian Cohen-Macaulay $R$-module.
\end{remark}

\begin{prop}
    \label{sgnorm}
    Let $T = D[x_1^{\pm 1}, \hdots, x_n^{\pm 1}]$ for some normal domain $D$, and let $R \subseteq T$ be a semigroup ring. Then the normalization $\tilde{R}$ of $R$ is a semigroup ring, which is pointed if and only if $R$ is pointed. If $D$ is excellent and $R$ is affine, then $\tilde{R}$ is affine.
\end{prop}

\begin{proof}
    A Laurent polynomial ring over a normal domain is a normal domain, so the normalization $\tilde{R}$ of $R$ is contained in $T$. In fact, $\tilde{R} = \text{frac}(R) \cap R^T$, where $R^T$ is the integral closure of $R$ in $T$. Note that $R^T$ is $\Z^n$-graded, by Theorem \ref{intclosSHthm}. The intersection $\text{frac}(R) \cap T$ is therefore $\Z^n$-graded as well\textemdash indeed, it is not hard to show the following lemma:
    \begin{lem}
    \label{fracgraded}
        Let $T = A[x_1^{\pm 1}, \hdots, x_n^{\pm 1}] $, and let  $R \subseteq T$ be a $\Z^n$\babelhyphen{nobreak}graded $A$-subalgebra, where $A$ is a domain. Then $\text{frac}(R) \cap T$ is also $\Z^n$-graded.
    \end{lem}
     The intersection of two compatibly $\Z^n$-graded rings is $\Z^n$-graded. Therefore $\tilde{R}$ is generated as a $D$-algebra by homogeneous elements.

    Let $c v$ be a general homogeneous element of $\tilde{R}$, where  $c \in D$ and $v$ is a monomial. I claim that $v \in \tilde{R}$. If $cv$ is integral over $R = D[M]$, then $v$ is integral over $L[M]$, where $L$ is the fraction field of $D$. But then some multiple of the exponent vector of $v$ lies in $M$; in other words, there is a $k \geq 1$ for which $v^k \in R$. Thus $v$ is integral over $R$. Moreover, if $cv$ is in the fraction field of $R$, then clearly $v = \frac{cv}{c} \in \text{frac}(R)$. Therefore $\tilde{R}= D[\tilde{M}]$ is a semigroup ring.

    It is not hard to show that the semigroups $M$ and $\tilde{M}$ generate the same real cone; so, $R$ is pointed if and only if $\tilde{R}$ is pointed.

    Suppose $D$ is excellent and $R$ is affine. Then $R$ is excellent as well, and therefore its normalization $\tilde{R}$ is a Noetherian $R$-module \cite[33.H Thm 78]{matsumura}. In particular, $\tilde{M}$ must be finitely generated as a semigroup, so $\tilde{R}$ is affine.
\end{proof}

\begin{thm}
    \label{semigroup1}
    Let $R = \Zp[M]$ be a pointed affine semigroup ring. Then the normalization $\tilde{R}$ of $R$ is a Noetherian Cohen-Macaulay $R$-module.
\end{thm}

\begin{proof}
    The coefficient ring $\Zp$ is an excellent Cohen-Macaulay normal domain. By Proposition \ref{sgnorm}, the normalization $\tilde{R}$ of $R$ is also a pointed affine semigroup ring. It follows from Theorem \ref{toricHthm1} that $\tilde{R} = \Zp[\tilde{M}]$ is a Cohen-Macaulay ring. Since $\tilde{R}$ is a module finite extension of $R$, it is therefore also Cohen-Macaulay as an $R$-module. 
    %{\color{red} (Note to self: see file "CM module finite lemma")}
\end{proof}

As discussed in Remark \ref{sgcontentproof}, this finishes the proof of Theorem \ref{sgcontent}. Armed with this result for semigroup rings, we can proceed to the general case.

\subsection{\texorpdfstring{$p$-semigroup rings}{p-semigroup rings}}
\label{psgsubsec}
In this section, let $T = \Q[x_1^{\pm 1}, \hdots, x_n^{\pm 1}]$ be the Laurent polynomial ring in $n$ variables. This ring can be thought of as a $\Z^n$-graded $\Zp$ algebra, where the homogeneous elements are $p$-monomials up to multiplication by a unit (see Remark \ref{homog}). All discussions from Section \ref{pmonomsec} apply in this setting.

\begin{defn}[$p$-semigroup rings]
    \label{psemigroupdef}
    A $\Zp$-subalgebra $R \subseteq T$ is called a \emph{\boldmath $p$-semigroup ring} if it is generated as a $\Zp$-algebra by $p$-monomials. Equivalently, $R$ is a $p$-semigroup ring if it has a $\Z^n$-grading given by multidegree in the variables $x_1, \hdots, x_n$.

    If $M \subset \Z^{n+1}$ is a semigroup containing $\enaught$, let $\Zp[\overline{M}]$ denote the $p$-semigroup ring generated by $p$-monomials $p^{t_0}x_1^{t_1}\hdots x_n^{t_n}$ for each $\ts \in M$. A $p$-semigroup ring $R = \Zp[\overline{M}]$ is called \emph{affine} if $M$ is a finitely generated semigroup, and \emph{pointed} if $\R_{\geq 0}M$ is a pointed cone.
\end{defn}

\begin{remark}
    \label{psgphi}
    With this notation, observe that the evaluation map $\phi(x_0) = p$ gives a surjection $\Zp[M] \to \Zp[\overline{M}]$ with kernel $(x_0 - p)\Zp[M]$ (see Thm \ref{phi}(c)). It follows from Proposition \ref{sgmaxhomog} that, for a pointed semigroup $M$ containing $\enaught$, the $p$-semigroup ring $\Zp[\overline{M}]$ has a unique maximal homogeneous ideal generated by its non-trivial $p$-monomials.
\end{remark}

The goal of this section is to prove that Conjecture \ref{sopcont} holds for pointed affine $p$-semigroup rings. The strategy will be the same as it was for semigroup rings: I will first show that the normalization of a pointed affine $p$-semigroup ring is still a pointed affine $p$-semigroup ring, and then I will argue that the normalization is Noetherian and Cohen-Macaulay.

The following two propositions show that $p$-semigroup rings behave well under normalization, analogously to their semigroup ring counterparts.

\begin{prop}
\label{intpmonom}
   Let $R \subseteq T$ be a $p$-semigroup ring. If $v$ is $p$-monomial, then v is integral over $R$ iff $v^k \in R$ for some $k \geq 1$.
\end{prop}
\begin{proof}
     The reverse implication is clear. If $v$ is integral over $R$, write
     \[
           v^m + f_1 v^{m-1} + \hdots + f_m = 0
     \]
     for some $f_i \in R$. Since $v$ is a $p$-monomial, the term $v^m$ is homogeneous. Write $v = p^t u$, where $u$ is some monomial in the variables $x_1, \hdots, x_n$. By projecting to the homogeneous component of $v^m$, we may assume that each coefficient $f_i \in R$ has the form $f_i = c_i p^{e_i}u^i$, where $c_i$ is either zero or a unit.
     
     Suppose that there is some $1 \leq k \leq m$ for which $c_k \neq 0$ and $kt \geq e_k$. Then
     \[
         v^k = p^{kt-e_k} \cdot \frac{1}{c_k}f_k \in R.
     \]
     So assume that this is not the case. Then, dividing out the monomial $u^m$ from the integral equation, we get
     \begin{align*}
         p^{mt} + c_1p^{e_1 + (m-1)t} + \hdots + c_m p^{e_m} &= 0\\
         1 + c_1p^{e_1 - t} + \hdots + c_mp^{e_m - mt} &= 0 \quad \text{where $c_i = 0$ or } e_i - it  > 0 \; \forall i\\
         c_1p^{e_1 - t} + \hdots + c_mp^{e_m - mt} &= -1 \notin p\Zp
     \end{align*}
     But whenever $c_i \neq 0$, we have by assumption that $it < e_i$, and therefore $c_ip^{e_i - it} \in p\Zp$ for every $i$. This is a contradiction.
\end{proof}

\begin{prop}
    \label{pmonomfracs}
    Let $R \subseteq T$ be a $p$-semigroup ring. Then a $p$-monomial $v$ is contained in the fraction field of $R$ if and only if $v = u/w$ for some $p$-monomials $u, w \in R$.
\end{prop}
\begin{proof}
    The backward implication is clear. Suppose $v \in \text{frac}(R)$; let $f, g \in R$ be nonzero Laurent polynomials such that $v = f/g$, or equivalently, $vg = f$. Choose any (nonzero) homogeneous component $cw$ of $g$, where $c \in \Zp^{\times}$ is a unit and $w$ is a $p$-monomial. Let $bu$ be the projection of $f$ to the homogeneous component of $v \cdot cw$, where $b \in \Zp^{\times}$ is a unit and $u$ is a $p$-monomial. Since $R$ is $\Z^n$-graded, we have $cw, bu \in R$ and
    \begin{align*}
        v \cdot cw = bu & &
        v = \frac{b}{c} \cdot \frac{u}{w},
    \end{align*}
    where $u/w$ is a $p$-monomial and $b/c$ is a unit. The expression of a homogeneous element as a unit times a $p$-monomial is unique, and $v$ is a homogeneous element with a coefficient of $1$. Therefore $b = c$ and in fact $v = u/w$, where $u,w \in R$ are $p$-monomials.
\end{proof}

\begin{thm}
    \label{psgnorm}
    Let $R \subseteq T$ be a pointed affine $p$-semigroup ring. Then the normalization $\tilde{R}$ of $R$ is also a pointed affine $p$-semigroup ring.
\end{thm}
\begin{proof}
    Let $R^T$ be the integral closure of $R$ in $T$. Then $R^T$ is $\Z^n$-graded by Theorem \ref{intclosSHthm}, and the normalization $\tilde{R}$ of $R$ is the intersection of $R^T$ with the fraction field of $R$. It follows from Lemma \ref{fracgraded} that $\tilde{R} = R^T \cap \text{frac}(R)$ is also $\Z^n$-graded, and thus generated by $p$-monomials.

    By Proposition \ref{intpmonom}, we get by the same argument used in Proposition \ref{sgnorm} that $R$ is pointed if and only if $\tilde{R}$ is pointed. Finally, $\tilde{R}$ is a Noetherian $R$-module, since $R$ is excellent \cite[33.H Thm 78]{matsumura}. Thus there is some finite set of $p$-monomials generating $\tilde{R}$ as a $\Zp$\babelhyphen{nobreak}algebra; in other words, $\tilde{R}$ is affine.
\end{proof}

\begin{cor}
\label{phitilde}
Let $M \subseteq \Z^{n+1}$ be a finitely generated pointed semigroup containing $\enaught$. Let $S = \Zp[M]$ and $R = \Zp[\overline{M}]$ with normalizations $\tilde{S}$ and $\tilde{R}$, respectively. Then the evaluation map $\phi$ induces a surjection $\tilde{S} \to \tilde{R}$, where the kernel is $(x_0 - p)\tilde{S}$.
\end{cor}
\begin{proof}
    By Theorems \ref{sgnorm} and \ref{psgnorm}, we have that $\tilde{S}$ is an affine semigroup ring and $\tilde{R}$ is an affine $p$-semigroup ring. Moreover, Propositions \ref{intpmonom} and \ref{pmonomfracs} show that the $p$-monomials in $\tilde{R}$ are in bijective correspondence with the monomials in $\tilde{S}$, and thus $\phi$ maps generators of $\tilde{S}$ to generators of $\tilde{R}$. The statement about the kernel follows directly from Theorem \ref{phi} (3).
\end{proof}

\begin{thm}
\label{psemigroup}
    Let $R \subseteq T$ be a pointed affine $p$-semigroup ring. Then the normalization $\tilde{R}$ of $R$ is a Noetherian Cohen-Macaulay $R$-module.
\end{thm}
\begin{proof}
    I have already remarked that $\tilde{R}$ is a Noetherian $R$-module, and that in fact $\tilde{R}$ is finitely generated as a $\Zp$-algebra by $p$-monomials (Thm \ref{psgnorm}). It remains to show that $\tilde{R}$ is a Cohen-Macaulay $R$-module. As in the proof of Theorem \ref{semigroup1}, it is enough to show that $\tilde{R}$ is Cohen-Macaulay as a ring.

    Write $R = \Zp[\overline{M}]$, where $M \subseteq \Z^{n+1}$ is the semigroup associated to $R$, and let $S = \Zp[M] \subset T[x_0^{\pm 1}]$ be the corresponding semigroup ring. By Corollary \ref{phitilde}, the normalization $\tilde{S}$ surjects onto $\tilde{R}$ via quotienting by the element $x_0 - p$, a nonzerodivisor contained in the homogeneous maximal ideal of $\tilde{S}$. Theorem \ref{semigroup1} shows that $\tilde{S}$ is Cohen-Macaulay, and the quotient of a Cohen-Macaulay graded ring by a nonzerodivisor in the homogeneous maximal idea is still Cohen-Macaulay (\cite[Exercise III.9]{gradedNO} and \cite[Thm 30 (ii)]{matsumura}).
\end{proof}

\begin{cor}
    Let $R \subseteq T$ be a pointed affine $p$-semigroup ring. Then any length $d$ sequence in $R$ with a minimal prime of height $d$ is a Q-sequence.
\end{cor}
\begin{proof}
    By Proposition \ref{reg}, it is enough to find a Noetherian Cohen Macaulay $R$\babelhyphen{nobreak}module, so the result follows immediately from from Theorem \ref{psemigroup}.
\end{proof}

\subsection{Non-pointed semigroups}
\label{locsec}
So far, I have proved Theorem \ref{bigsec5thm} for the case of pointed affine $p$-semigroup rings, by demonstrating that the normalization of a pointed affine $p$-semigroup ring is a Noetherian Cohen-Macaulay module. The goal of Section 5.3 is to show that the theorem holds for non-pointed $p$-semigroup rings as well. I will accomplish this by proving that every affine $p$-semigroup ring is the localization of a pointed affine $p$-semigroup ring.

First, I will need the following proposition:

%{\color{red} TODO: Maybe redo Thm \ref{localizingnormal} to take fuller advantage of Gordon's lemma?}

\begin{prop}[Gordon's lemma, {\cite[Thm 7.16]{cca}}]
    \label{latticecone}
    Let $\sigma \subseteq \R^n$ be a rational polyhedral cone, and let $G \subset \Z^n$ be a subgroup. Then $G \cap \sigma$ is a finitely generated semigroup.
\end{prop}

\begin{thm}
        \label{localizingnormal}
        Let $M$ be a finitely generated normal semigroup. Then there is a finitely generated sub-semigroup $N$ generating a pointed cone and an element $\us \in N$ such that $M = N + \Z \us$. Moreover, if $\bs \in M$ is any single fixed element, then $N$ can be chosen such that $\bs \in N$.
\end{thm}
\begin{proof}
     Fix an embedding $M \subset \Z^n$ such that $\Z M = \Z^n$. Let $\tau = \R_{\geq 0} M$ be the rational polyhedral cone generated by $M$. Since $M$ is normal, we have $M = \tau \cap \Z^n$ \cite[Prop 7.25]{cca}.

    Consider the vector space $V = \{ \vs \in \R^n \; : \; \R \vs \subseteq \tau \}$. For any $\vs \in V \cap \Z^n$, notice that $\vs, -\vs \in \tau \cap \Z^n = M$. Therefore, let $\vs_1, \hdots, \vs_k$ be a basis for $V$ such that $\vs_i, -\vs_i \in M$ for every $i$.

    I claim that there are $\ws_1, \hdots, \ws_t \in M$ such that $\vs_1, \hdots, \vs_k, \ws_1, \hdots, \ws_t$ is a basis for $\R^n$. If $V = \R^n$, then we're done ($t = 0$). Otherwise, $V \subset \R_{\geq 0} M$ is a proper subset. By density of rationals, there exists $\ws_1 \in \Q_{\geq 0} M$ with $\ws_1 \notin V$, and by clearing denominators we may assume $\ws_1 \in M$.

    If $V + \R_{\geq 0} \ws_1 = \R_{\geq 0} M$, then $V + \R \ws_1 = \R M = \R^n$, and $\vs_1, \hdots, \vs_k, \ws_1$ is a basis. Otherwise, by the same argument, there is $\ws_2 \in M$ with $\ws_2 \notin V + \R_{\geq 0} \ws_1$. I claim that $\ws_2$ is independent from $\vs_1, \hdots, \vs_k, \ws_1$. For contradiction, assume that $\ws_2 \in V + \R \ws_1$; since $\ws_2 \notin V + \R_{\geq 0} \ws_1$, we can write $\ws_2 = \vs - r_1 \ws_1$ for some $\vs \in V$ and $r_1 \in \R_{\geq 0}$. Then $-\ws_2 = -\vs + r_1 \ws_1 \in V + \R_{\geq 0} \ws_1 \subset \R_{\geq 0} M$, so $\R \ws_2 \subseteq \R_{\geq 0} M$. But this is impossible, since $\ws_2 \notin V$.

    By induction, we can choose $\ws_3, \hdots, \ws_t \in M$ such that $\ws_{i+1} \notin V + \R (\ws_1, \hdots, \ws_i)$ and $V + \R(\ws_1, \hdots, \ws_t) = \R^n$. Then $\vs_1, \hdots, \vs_k, \ws_1, \hdots, \ws_t$ is a basis for $\R^n$.

    Let $\pi: \R^n \to V$ be the projection map onto $V$ with respect to this basis. Then $\ker \pi$ is a rational polyhedral cone, so $\tau \cap \ker \pi$ is also a rational polyhedral cone with some generators $\ws_1', \hdots, \ws_{t'}' \in \Z^n \cap \tau = M$. Let $\sigma = \R_{\geq 0} (\vs_1, \hdots, \vs_k, \ws_1', \hdots, \ws_{t'}')$. It is not hard to show that linearly independent vectors generate a pointed cone, so $\sigma$ is pointed. Take $N = M \cap \sigma$ and $\us = \vs_1 + \hdots + \vs_k \in N$. The semigroup $N$ generates a pointed cone $\R_{\geq 0} N = \sigma$. Moreover, $N$ is finitely generated by Gordon's lemma, since
    \[
       N = M \cap \sigma = \Z^n \cap (\tau \cap \sigma). 
    \]
    It remains to show that $M = N + \Z \us$. Clearly $N + \Z \us \subseteq M$; on the other hand, if $\as \in M$, write $\as = \vs + \ws$, where $\vs \in V$ and $\ws \in \ker \pi$. Then $-\vs \in V \subseteq \tau$, so $\ws = \as - \vs \in \tau \cap \ker \pi = \R_{\geq 0} (\ws_1', \hdots, \ws_{t'}')$. Let $z \in \N$ be large enough such that $\vs + z\us \in \R_{\geq 0} (\vs_1, \hdots, \vs_k)$. Then
    \[
       \as + z\us = (\vs + z\us) + \ws \in M \cap \sigma = N.
    \]
    Therefore $\as \in N + \Z \us$.

    Let $\bs \in M$ be some fixed element. If $\bs \in V$, there is a basis for $V$ with $\vs_1 = \bs$, so that $\bs \in M \cap \R_{\geq 0}(\vs_1, \hdots, \vs_k) \subset N$. Otherwise, take $\ws_1 = \bs \in M \setminus V$. Then $\bs \in M \cap \ker \pi \subset N$.
\end{proof}

\begin{thm}
    \label{localizing}
    Let $M$ be any finitely generated semigroup, not necessarily normal. Then there is a finitely generated sub-semigroup $N$ with $\R_{\geq 0}N$ pointed and an element $\us \in N$ such that $M = N + \Z \us$, and $N$ can be chosen such that $\bs \in N$ for any single fixed element $\bs \in M$. Consequently, any affine semigroup or $p$-semigroup ring is the localization of a pointed affine semigroup or $p$-semigroup ring.
\end{thm}
\begin{proof}
    Let $\tilde{M}$ be the normalization of $M$; this is finitely generated, since $K[M]$ is excellent for any field $K$ \cite[33.H Thm 78]{matsumura}. By Theorem \ref{localizingnormal}, there is a pointed affine sub-semigroup $\tilde{N} \subseteq \tilde{M}$ and $\tilde{\us} \in \tilde{N}$ such that $\tilde{M} = \tilde{N} + \Z \tilde{\us}$.

    I claim that $\tilde{\us}$ can be replaced with some $\us \in \tilde{N} \cap M$. The monomial $\xs^{\tilde{\us}}$ is integral over $K[M]$, so there is a nonzero $t \in \N$ such that $t\tilde{u} = \us \in M$. Then
    \[
        -\tilde{\us} =  (t-1)\tilde{\us} - \us \in \tilde{N} + \Z \us,
    \]
    and therefore $\tilde{N} + \Z \tilde{\us} = \tilde{N} + \Z \us = \tilde{M}$.
    
    Notice that $-\us \in M$ also; certainly $-\us \in \tilde{M}$, so $\xs^{-\us}$ is integral over $K[M]$ and there is some $t>0$ such that $t (-\us) \in M$. Then $-t\us + (t-1)\us = -\us \in M$.

    Let $N = M \cap \tilde{N}$. Clearly $M \supseteq N + \Z \us$; on the other hand, if $\vs \in M$, write $\vs = \tilde{\ws} + z \us$ for some $z \in \Z$, $\tilde{\ws} \in \tilde{N}$. Then $\tilde{\ws} = \vs - z \us \in M \cap \tilde{N} = N$. Therefore, we have $M = N + \Z \us$. The cone $\R_{\geq 0} N$ is certainly pointed, since it is contained in $\R_{\geq 0} \tilde{N}$. It remains to show that $N$ is finitely generated.

    I claim that the inclusion map $K[N] \hookrightarrow K[\tilde{N}]$ is module-finite. Since $\tilde{N}$ is finitely-generated, $K[\tilde{N}]$ is in particular a finitely generated $K[N]$-algebra; therefore it is enough to show that $K[N] \hookrightarrow K[\tilde{N}]$ is integral. Let $\tilde{\ws} \in \tilde{N}$. By integrality over $K[M]$, there exists $t \in \N$ nonzero such that $t \tilde{\ws} \in M$. But then $t \tilde{\ws} \in \tilde{N} \cap M = N$, so $\xs^{\tilde{\ws}}$ is integral over $K[N]$. 

    The Eakin-Nagata theorem says that if $A \hookrightarrow B$ is module finite and $B$ is Noetherian, then $A$ is Noetherian \cite[Thm 3.7 (i)]{matsumura}. The semigroup ring $K[\tilde{N}]$ is Noetherian, so by Eakin-Nagata it follows that $K[N]$ is Noetherian; in particular, the homogeneous maximal ideal is finitely generated! Therefore $N$ is a finitely generated semigroup.

    We have shown that $\Zp[M] = \Zp[N][\xs^{-\us}]$, where $\Zp[N]$ is a pointed affine semigroup ring. To get the corresponding result for $p$-semigroup rings, we will need $\enaught$ to be contained in $N$. Indeed, if $\bs \in M$ is any single fixed element, $\tilde{N}$ can be chosen to contain $\bs$ (Thm \ref{localizingnormal}), in which case $\bs \in M \cap \tilde{N} = N$. Therefore, if $\enaught \in M$, then $N$ can be chosen to contain $\enaught$, in which case $\Zp[\overline{M}] = \Zp[\overline{N}][p^{-u_0}x_1^{-u_1}\hdots x_n^{-u_n}]$.
\end{proof}

\begin{cor}
    If $R$ is an affine semigroup or $p$-semigroup ring, then the normalization $\tilde{R}$ of $R$ is a Noetherian Cohen-Macaulay $R$ module. Therefore, any length $d$ sequence with a minimal prime of height $d$ in $R$ is a Q-sequence. 
\end{cor}
\begin{remark}
    This finishes the proof of Theorem \ref{bigsec5thm}.
\end{remark}
\begin{proof}
    Affine ($p$-)semigroup rings are excellent, so $\tilde{R}$ is a Noetherian $R$-module \cite[33.H Thm 78]{matsumura}. By Theorem \ref{localizing}, there is a pointed affine ($p$-)semigroup ring $S$ and an element $f \in S$ such that $R = S[1/f]$. Let $\tilde{S}$ be the normalization of $S$. Then $\tilde{R} = \tilde{S}[1/f]$, since normalization commutes with localization. By Theorems \ref{semigroup1} and \ref{psemigroup}, we know that $\tilde{S}$ is Noetherian and Cohen-Macaulay (both as a ring and as an $S$-module). Any localization of a Cohen-Macaulay ring is still Cohen-Macaulay, so $\tilde{R}$ a Cohen-Macauly ring which is module finite over $R$, and thus Cohen-Macaulay as an $R$-module.

    Since $R$ has a Noetherian Cohen-Macaulay module, any length $d$ sequence with a minimal prime of height $d$ in $R$ is a Q-sequence by Proposition \ref{reg}.
\end{proof}

\section{Toric face rings}
A \textit{toric face ring} is another combinatorial ring, typically defined over a field, which is constructed by “gluing" finitely many affine semigroup rings along faces of the associated cones. In the last section of this paper, I will show that toric face rings over $\Zp$ satisfy Conjecture \ref{sopcont}. I will then construct \textit{$p$-toric face rings} by gluing together $p$-semigroup rings in an analogous fashion, and show that Conjecture \ref{sopcont} holds for these as well.

\begin{defn}[Algebra retract, {\cite[Def 1.1]{BCK+23}}]
    \label{retracts}
    Let $R$ be an $A$-algebra, where $A$ is a domain. An inclusion of $A$-algebras $\iota: S \hookrightarrow A$ is called an \emph{algebra retract} if there is a surjective $A$-algebra map $\pi: R \twoheadrightarrow S$ such that $\pi \circ \iota = \text{id}_S$.

    If $\{S_j\}_{j \in J}$ is a finite collection of domains which are algebra retracts of $R$, with defining maps $\iota_j: S_j \hookrightarrow A$ and $\pi_j: R \twoheadrightarrow S_j$, then $R$ is \emph{realized by the retracts} $\{S_j\}_{j \in J}$ if the following two conditions hold:
    \begin{enumerate}
        \item The $A$-algebra map $R \to \prod_{j} S_j$ given by $f \mapsto (\pi_j(f))_j$ is injective.

        \item The $S_j$ are irredundant, in the sense that the map in (1) would not be injective if any of the $S_j$ were dropped.
    \end{enumerate}
\end{defn}

\begin{remark}
    If condition (1) is satisfied for some collection $\{S_j\}_{j \in J}$, then there is some subset which is irredundant.
\end{remark}

\begin{remark}
    \label{retractass}
    Since the $S_j$ are domains, the kernels $\ker (\pi_j) =: P_j$ are prime ideals of $R$. Conditions (1) and (2) mean exactly that $0 = \bigcap_{j} P_j$
    is the irredundant primary decomposition of the zero ideal. In particular, $R$ is a reduced ring, and the associated primes are $\Ass(R) = \{P_j\}_{j \in J}$.
\end{remark}

\begin{prop}
    \label{retractQseq} Let $R$ be an $A$-algebra and let $\{S_j\}_{j \in J}$ be a finite collection of domains which are $A$-algebra retracts of $R$ such that $R$ is realized by $\{S_j\}_{j \in J}$. Suppose that every $S_j$ satisfies Conjecture \ref{sopcont}; in other words, any length $d$ sequence in $S_j$ generating an ideal of height $d$ is a $Q$-sequence in $S_j$. Then $R$ also satisfies Conjecture \ref{sopcont}.
\end{prop}
\begin{proof}
    Follows immediately from Proposition \ref{minprimeQseq} and Remark \ref{retractass}.
\end{proof}

\begin{defn}[Toric face ring, {\cite[page 202]{stanleyTFR}}]
    \label{toricfacering}
    Let $\Sigma \subset \R^n$ be a rational polyhedral fan of pointed cones (see \cite[pages 9, 20]{fulton}). A \emph{monoidal complex} supported on $\Sigma$ is a collection $\M = \{M_\sigma : \sigma \in \Sigma \}$ of finitely generated semigroups such that
    \begin{enumerate}[a)]
        \item $M_\sigma \subset \sigma$ and $\R_{\geq 0}M = \sigma$ for every $\sigma \in \Sigma$.
        \item If $\sigma$ is a face of $\tau \in \Sigma$, then $M_\sigma = \sigma \cap M_\tau$.
    \end{enumerate}
    Define $|\M| := \bigcup M_\sigma$ as the set union in $\Z^n$ of the $M_\sigma$ over all $\sigma \in \Sigma$. For a domain $A$, the \emph{toric face ring} defined by $\M$ over $A$ is given as an $A$-module by
    \[
        A[\M] := \bigoplus_{\ts \in |\M|} A x_1^{t_1} \hdots x_n^{t_n},
    \]
    with the ring structure induced by the rule
    \[
    \xs^{\as}\cdot \xs^{\bs} := \begin{cases}
        \xs^{\as + \bs} & \text{if } \exists \sigma \in \Sigma \;\; \as, \bs \in M_\sigma\\
        0 & \text{otherwise}
    \end{cases}
    \]
\end{defn}

\begin{remark}
    \label{toricretract}
    I have expanded the definition slightly from the version in \cite{stanleyTFR} to allow the coefficients to be a domain instead of a field.
    
    As discussed in \cite[page 5]{BCK+23}, toric face rings are realized by retracts $\{A[M_\tau] : \tau \in \mathscr{F}(\Sigma)\}$, where $\mathscr{F}(\Sigma)$ is the set of maximal cones. The kernel of the surjection $A[\M] \to A[M_\tau]$ is the homogeneous prime $P_\tau := (\xs^{\ts} : \ts \notin M_\tau) A[\M]$.
\end{remark}

\begin{thm}
    Let $R = \Zp[\M]$ be a toric face ring. Then any length $d$ sequence in $R$ with a minimal prime of height $d$ is a $Q$-sequence.
\end{thm}
\begin{proof}
    $R$ is realized by retracts which are pointed affine semigroup rings (Remark \ref{toricretract}). I have proved in the previous section (Theorem \ref{sgcontent}) that pointed affine semigroup rings satisfy Conjecture \ref{sopcont}. It follows from Proposition \ref{retractQseq} that $R$ satisfies Conjecture \ref{sopcont} as well.
\end{proof}

%%%%%%%%%%%%%%%%%%%%%%%%%%%%%%%
I would now like to extend the notion of toric face rings, analogously to the extension of semigroup rings to $p$-semigroup rings in the previous section. The resulting class of rings, which I will call \emph{\boldmath $p$-toric face rings}, will be realized by retracts which are $p$-semigroup rings.  It is not immediately obvious, however, how this should be done. 

Consider a monoidal complex $\M \subseteq \Z^{n+1}$ on a rational polyhedral fan $\Sigma$, where $\Z^{n+1}$ has basis vectors $\enaught, \underline{e_1}, \hdots, \underline{e_n}$. The goal is to define a \emph{\boldmath $p$-toric face ring} $\Zp[\overline{\M}]$ with 
\[
   p^{t_0} x_1^{t_1} \hdots x_n^{t_n} \in \Zp[\M] \quad \text{whenever} \quad \ts \in |\M|,
\]
so that $\Zp[\overline{\M}]$ is realized by $\Zp$-algebra retracts $\{\Zp[\overline{M_\tau}] : \tau \in \mathscr{F}(\Sigma) \}$. However, the fact that $p$ is an element of the coefficient ring causes some issues.

First of all, the interactions between the retracts $\Zp[\overline{M_\tau}] \hookrightarrow \Zp[\overline{\M}]$ should be totally accounted for by the structure of the monoidal complex $\M$. Therefore we have to assume that $\enaught \in M_\tau$ for any maximal cone $\tau$ of $\Sigma$, since otherwise $M_\tau$ will not capture all of the $p$-monomials in $\Zp[\overline{M_\tau}]$. This imposes a rather heavy restriction on the fans $\Sigma$ which are used for this construction.

Secondly, one must be more careful indexing the direct sum used to define $\Zp[\overline{\M}]$ as a $\Zp$-module, because distinct $p$-monomials can span overlapping $\Zp$-modules. Namely, one has that
\[
   \Zp p^{b_0} x_1^{b_1} \hdots x_n^{b_n} \subseteq \Zp p^{a_0} x_1^{a_1} \hdots x_n^{a_n} \quad \text{iff} \quad \bs \in \as + \N \enaught.
\]

Each $p$-monomial should appear in at most one direct summand, so we cannot index the sum over $|\M| = \bigcup_{\sigma \in \Sigma} M_\sigma$ as in Definition \ref{toricfacering}. I will argue that in each maximal cone $\tau$, we can choose some subset of $M_\tau$ that will irredundantly capture all the $p$-monomials in $\Zp[\overline{M_\tau}]$, and moreover that this choice will be compatible with all other maximal cones in the fan.

\begin{prop}
    \label{minrep}
    Let $M \subset \Z^{n+1}$ be a semigroup generating a pointed positive real cone, and assume that $\enaught \in M$. There is always a minimal representative in $M$ for any equivalence class of $\Z M/\Z\enaught$, in the sense that for any $\bs \in M$, there exists a unique $\as \in M$ such that $\as + \N \enaught = M \cap (\bs + \Z \enaught)$.
\end{prop}
\begin{proof}
    Let $\pi_0 : \Z^{n+1} \to \Z$ be the zeroth projection map. I claim that the image $\pi_0(M \cap (\bs + \Z \enaught))$ is bounded below. If not, that would imply that $\bs + \Z \enaught$ is contained in $M$. But $\R_{\geq 0}M$ is pointed, so this cannot be the case.

    Choose $\as \in M \cap (\bs + \Z \enaught)$ so that $\pi_0(\as)$ is minimal. Then any other element of $M \cap (\bs + \Z \enaught)$ has a coordinate in $\enaught$ that is greater than or equal to $\pi_0(\as)$, and therefore $\as + \N \enaught = M \cap (\bs + \Z \enaught)$. Note that if $\as + \N\enaught = \bs + \N \enaught$, then $\as = \bs$; thus the choice of $\as$ is unique.
\end{proof}

\begin{prop}
    Let $\Sigma$ be a polyhedral rational fan of pointed cones, and let $\M$ be a monoidal complex supported on $\Sigma$ such that $\enaught \in M_\tau$ for every maximal cone $\tau \in \mathscr{F}(\Sigma)$. Let $\sigma, \tau \in \mathscr{F}(\Sigma)$. If $\as \in M_\sigma$, $\bs \in M_\tau$ are minimal in the sense of Proposition \ref{minrep}, and $\as + \Z\enaught = \bs + \Z \enaught$, then $\as = \bs$.
\end{prop}
\begin{proof}
    It is enough to show that $\as$ and $\bs$ are both in $\M_{\sigma \cap \tau}$, by uniqueness of the minimal representative in Lemma \ref{minrep}. The intersection $\sigma \cap \tau$ is a a face of both $\sigma$ and $\tau$. Therefore, $H \cap \sigma = H \cap \tau = \sigma \cap \tau$ for some hyperplane $H$.

    Assume without loss of generality that $\bs \in \as + \N \enaught$. Since $\as, \enaught \in M_\sigma$, this implies that $\bs \in M_\sigma \cap M_\tau = M_{\sigma \cap \tau} \subset H$. Clearly $\enaught \in H$, so $\as \in \bs + \Z \enaught \subset H$. But then $\as \in H \cap \sigma = H \cap \tau$, and so $\as \in M_\sigma \cap \tau = M_{\sigma \cap \tau}$ as well. 
\end{proof}

We are now ready to define $p$-toric face rings:
\begin{defn}
    Let $\Sigma \subset \R^{n+1}$ be a polyhedral rational fan of pointed cones, and let $\M$ be a monoidal complex supported on $\Sigma$ such that $\enaught \in M_\tau$ for every maximal cone $\tau \in \mathscr{F}(\Sigma)$. For each maximal cone $\tau \in \mathscr{F}(\Sigma)$, let $|M_\tau|_0 \subset M_\tau$ be the collection of minimal representatives of equivalence classes in $\Z M/ \Z \enaught$ in the sense of Lemma \ref{minrep}. Define $|\M|_0$ to be the union of the $|M_\tau|_0$. Then the \emph{\boldmath $p$-toric face ring} defined by $\M$ is given as a $\Zp$-module by
    \[
       \Zp[\overline{\M}] := \bigoplus_{\ts \in |\M|_0} \Zp p^{t_0} x_1^{t_1} \hdots x_n^{t_n},
    \]
    with the ring structure induced by the rule
    \[
    p^{a_0} x_1^{a_1} \hdots x_n^{a_n} \cdot p^{b_0} x_1^{b_1} \hdots x_n^{b_n} := \begin{cases}
        p^{a_0+b_0} x_1^{a_1+b_1} \hdots x_n^{a_n+b_n} & \text{if } \exists \sigma \;\; \as, \bs \in M_\sigma\\
        0 & \text{otherwise}
    \end{cases}
    \]
\end{defn}

\begin{ex}
    Let $\tau_i = \R_{\geq 0} M_{\tau_i}$ for semigroups
    \begin{align*}
        M_{\tau_1} &= \Span_{\N} \left\{ \begin{pmatrix}
            1\\0\\0
        \end{pmatrix}, \begin{pmatrix}
            2\\3\\0
        \end{pmatrix}, \begin{pmatrix}
            2\\2\\2
        \end{pmatrix} \right\} &
        M_{\tau_2} &= \Span_{\N} \left\{ \begin{pmatrix}
            1\\0\\0
        \end{pmatrix}, \begin{pmatrix}
            2\\2\\2
        \end{pmatrix}, \begin{pmatrix}
            0\\0\\3
        \end{pmatrix} \right\}
    \end{align*}
    Then $\tau_1, \tau_2$ are the faces of a polyhedral rational fan $\Sigma \subset \R^3$ and these $M_{\tau_i}$ induce a monoidal complex $\M$ on $\Sigma$. In this case, the calculation of $|\M|_0$ is not difficult. One can check that
    \begin{align*}
        |M_{\tau_1}|_0 &= \Span_{\N} \left\{ \begin{pmatrix}
            2\\3\\0
        \end{pmatrix}, \begin{pmatrix}
            2\\2\\2
        \end{pmatrix} \right\} &\text{and }
        |M_{\tau_2}|_0 &= \Span_{\N} \left\{ \begin{pmatrix}
            2\\2\\2
        \end{pmatrix}, \begin{pmatrix}
            0\\0\\3
        \end{pmatrix} \right\}.
    \end{align*}
    The $p$-toric face ring defined by $\M$ is therefore
    \begin{align*}
        \Zp[\overline{\M}] &= \bigoplus_{\ts \in |\M|_0} \Zp p^{t_0}x^{t_1}y^{t_2}\\
        &= \begin{matrix}
             & \Zp p^2 x^3 & \oplus & \Zp p^4 x^6     & \oplus & \Zp p^4 x^5 y^2 & \oplus & \Zp p^6 x^9     & \oplus & \hdots \\
         \oplus & \Zp      & \oplus & \Zp p^2 x^2 y^2 & \oplus & \Zp p^4 x^4 y^4 & \oplus & \Zp p^6 x^6 y^6 & \oplus & \hdots\\
         \oplus & \Zp y^3  & \oplus & \Zp y^6         & \oplus & \Zp p^2 x^2 y^5 & \oplus & \Zp y^9 & \oplus & \hdots 
        \end{matrix}
    \end{align*}
    I have arranged the sum so that the bottom row is $P_{\tau_1} = (p^{t_0}x^{t_1}y^{t_2} : \ts \notin M_{\tau_1})$, and the top row is $P_{\tau_2} = (p^{t_0}x^{t_1}y^{t_2}  : \ts \notin M_{\tau_2})$. The embedding of $\Zp[\overline{M_{\tau_1}}]$ into $\Zp[\overline{\M}]$ is given by the top two rows of the direct sum, and the embedding of $\Zp[\overline{M_{\tau_2}}]$ into $\Zp[\overline{\M}]$ is given by the bottom two rows. We get that $\Zp[\overline{M_{\tau_i}}] \hookrightarrow \Zp[\overline{\M}] \twoheadrightarrow \Zp[\overline{\M}]/P_{\tau_i} = \Zp[\overline{M_{\tau_i}}]$ are $\Zp$-algebra retracts for each $i$, and $P_{\tau_1} \cap P_{\tau_2} = 0$. Therefore $\Zp[\overline{M}]$ is realized by retracts which are $p$-semigroup rings.
\end{ex}

I will now examine the structure of $p$-toric face rings in general. We shall see that $p$-toric face rings are the images of toric face rings under evaluation of $x_0$ at $p$, and that the algebra retract structure descends under this map. This will allow us to conclude that Conjecture \ref{sopcont} holds for $p$-toric face rings.

\begin{prop}
    \label{toricfacekernel}
    Let $\Sigma$ be a polyhedral rational fan of pointed cones, and let $\M$ be a monoidal complex supported on $\Sigma$ such that $\enaught \in M_\tau$ for every maximal cone $\tau \in \mathscr{F}(\Sigma)$. Let $\phi$ be the $\Zp$-module map
    \[
         \phi: \Zp[\M] \twoheadrightarrow \Zp[\overline{\M}] \qquad \xs^{\ts} \mapsto p^{t_0}x_1^{t_1}\hdots x_n^{t_n}.
    \]
    Then $\phi$ is in fact a ring homomorphism with kernel $(x_0 - p)\Zp[\M]$.
\end{prop}
\begin{proof}
    To show that evaluation of $x_0$ at $p$ is a ring homomorphism on $\Zp[\M]$, it is enough to check that $\phi(1) = 1$ and $\phi$ respects multiplication of the $\Zp$\babelhyphen{nobreak}module generators $\{\xs^{\ts} \; : \; \ts \in |\M| \}.$ These are both relatively clear. It remains only to check that $\ker \phi = (x_0 - p)\Zp[\M]$.
    
     Let $T := \Zp[x_0^{\pm 1}, \hdots, x_n^{\pm 1}]$. I claim that the natural inclusion of $\Zp$-modules $\iota : \Zp[\M] \hookrightarrow T$ is in fact a $\Zp[x_0]$-module morphism. Indeed, $\Zp[\M]$ is generated as a $\Zp$\babelhyphen{nobreak}module by monomials $\xs^{\ts}$ where $\ts \in M_\tau$ for some maximal cone $\tau$, and since $\enaught \in M_\tau$ it follows that
    \[
       \iota (x_0 \xs^{\ts}) = \iota (\xs^{\ts + \enaught}) = \xs^{\ts + \enaught} = x_0 \iota(\xs^{\ts}).
    \]
    Remark that $\ker\phi \subset \Zp[\M]$ is an ideal, so in particular it is a $\Zp[x_0]$\babelhyphen{nobreak}submodule. Viewing $\phi$ as a map of $\Zp$-modules, it is clear that $\iota(\ker \phi) = (x_0 - p)T \cap \iota(\Zp[\M])$. In Theorem \ref{phi}(3), I showed that the kernel of $\phi$ restricted to a semigroup ring containing $x_0$ is the principal ideal generated by $(x_0 - p)$. The proof, however, used only the $\Zp[x_0]$-module structure! Therefore, by identical reasoning, it follows that
    \[
       \iota(\ker \phi) = (x_0 - p) \iota(\Zp[\M]) = \iota((x_0 - p)\Zp[\M]),
    \]
    and so $\ker \phi = (x_0 - p)\Zp[\M]$.
\end{proof}

\begin{prop}
    \label{ptoricretract}
    Let $\Sigma$ be a polyhedral rational fan of pointed cones, and let $\M$ be a monoidal complex supported on $\Sigma$ such that $\enaught \in M_\tau$ for every maximal cone $\tau \in \mathscr{F}(\Sigma)$. Then the $p$-toric face ring $\Zp[\overline{\M}]$ is realized by retracts $\Zp[\overline{M_\tau}]$ for $\tau \in \mathscr{F}(\Sigma)$.
\end{prop}
\begin{proof}
    For every $\tau \in \mathscr{F}(\Sigma)$, there are algebra maps
    \begin{center}
        \begin{tikzcd}
	{\Zp[M_\tau]} & {\Zp[\M]} & {\Zp[M_\tau]}
	\arrow["{\iota_\tau}", hook, from=1-1, to=1-2]
	\arrow["{\pi_\tau}", two heads, from=1-2, to=1-3]
        \end{tikzcd},
    \end{center}
    where the kernel of this surjection is $P_\tau = (\xs^{\ts} : \ts \notin M_\tau) \Zp[\M]$. Let $\phi$ denote the evaluation map $x_0 \mapsto p$, and let $\phi_\tau := \phi|_{\Zp[M_\tau]}$. We have $\ker(\phi \circ \iota_\tau) = \ker(\phi_\tau)$, so there is an injective algebra map $\overline{\iota_\tau}$ making the diagram commute:
    \[
       \begin{tikzcd}[ampersand replacement=\&]
	{\Zp[M_\tau]} \& {\Zp[\M]} \\
	{\Zp[\overline{M_\tau}]} \& {\Zp[\overline{\M}]}
	\arrow["{\iota_\tau}", hook, from=1-1, to=1-2]
	\arrow["{\overline{\iota_\tau}}"', dashed, hook, from=2-1, to=2-2]
	\arrow["{\phi_\tau}"', two heads, from=1-1, to=2-1]
	\arrow["\phi", two heads, from=1-2, to=2-2]
       \end{tikzcd}
    \]
    Since $\ker \phi_\tau = (x_0 - p)\Zp[M_\tau]$, we also have
    \[
        \ker(\phi_\tau \circ \pi_\tau) = P_\tau + (x_0-p)\Zp[\M] \supset \ker \phi \qquad \text{by Proposition \ref{toricfacekernel}},
    \]
    and therefore there is an algebra map $\overline{\pi_\tau}$ making the diagram commute:
    \[
        \begin{tikzcd}[ampersand replacement=\&]
	{\Zp[\M]} \& {\Zp[M_\tau]} \\
	{\Zp[\overline{\M}]} \& {\Zp[\overline{M_\tau}]}
	\arrow["{\pi_\tau}", two heads, from=1-1, to=1-2]
	\arrow["{\overline{\pi_\tau}}"', dashed, from=2-1, to=2-2]
	\arrow["{\phi_\tau}", two heads, from=1-2, to=2-2]
	\arrow["\phi"', two heads, from=1-1, to=2-1]
        \end{tikzcd}
    \]
    Note that $\overline{\pi_\tau}$ must be surjective, since $\overline{\pi_\tau} \circ \phi = \phi_\tau \circ \pi_\tau$ is surjective. Therefore $\Zp[\overline{M_\tau}]$ is an algebra retract of $\Zp[\M]$ for every $\tau \in \mathscr{F}(\Sigma)$.

    Finally, note that $\ker \overline{\pi_\tau} = (P_\tau + \ker \phi )/ \ker \phi$, so
    \begin{align*}
        \bigcap_{\tau \in \mathscr{F}(\Sigma)} \ker \overline{\pi_\tau} &= \frac{ \bigcap (P_\tau + \ker \phi) }{\ker \phi} = \frac{(\bigcap P_\tau ) + \ker \phi}{\ker \phi} = \frac{\ker \phi}{\ker \phi}= 0
    \end{align*}
    is the irredundant primary decomposition of $0$ in $\Zp[\overline{\M}]$.
\end{proof}

\begin{thm}
    \label{ptoric}
    Let $R = \Zp[\overline{\M}]$ be a $p$-toric face ring. Then $R$ satisfies Conjecture \ref{sopcont}, in the sense that any length $d$ sequence in $R$ with a minimal prime of height $d$ is a Q-sequence.
\end{thm}
\begin{proof}
    I have shown that $R$ is realized by retracts which are affine $p$-semigroup rings (Proposition \ref{ptoricretract}), and that affine $p$-semigroup rings satisfy Conjecture \ref{sopcont} (Theorem \ref{bigsec5thm}). Therefore, by Proposition \ref{retractQseq}, it follows that $R$ satisfies Conjecture \ref{sopcont} as well.
\end{proof}

\bibliographystyle{abbrv}
\bibliography{sections/cas_refs}

\begin{thebibliography}{10}

\bibitem{MR3814651}
Y.~Andr\'{e}.
\newblock La conjecture du facteur direct.
\newblock {\em Publ. Math. Inst. Hautes \'{E}tudes Sci.}, 127:71--93, 2018.

\bibitem{BCK+23}
C.~Berkesch, C.-Y.~J. Chan, P.~Klein, L.~F. Matusevich, J.~Page, and J.~Vassilev.
\newblock Differential operators, retracts, and toric face rings, 2023.

\bibitem{fulton}
W.~Fulton.
\newblock {\em Introduction to toric varieties}, volume 131 of {\em Annals of Mathematics Studies}.
\newblock Princeton University Press, Princeton, NJ, 1993.
\newblock The William H. Roever Lectures in Geometry.

\bibitem{toricH}
M.~Hochster.
\newblock Rings of invariants of tori, {C}ohen-{M}acaulay rings generated by monomials, and polytopes.
\newblock {\em Ann. of Math. (2)}, 96:318--337, 1972.

\bibitem{contentHH}
M.~Hochster and C.~Huneke.
\newblock Quasilength, latent regular sequences, and content of local cohomology.
\newblock {\em J. Algebra}, 322(9):3170--3193, 2009.

\bibitem{intclosSH}
S.~Irena and H.~Craig.
\newblock {\em Integral Closure of Ideals, Rings, and Modules.}
\newblock Number Vol. 336 in London Mathematical Society Lecture Note Series. Cambridge University Press, 2006.

\bibitem{matsumura}
H.~Matsumura.
\newblock {\em Commutative algebra}, volume~56 of {\em Mathematics Lecture Note Series}.
\newblock Benjamin/Cummings Publishing Co., Inc., Reading, MA, second edition, 1980.

\bibitem{cca}
E.~Miller and B.~Sturmfels.
\newblock {\em Combinatorial commutative algebra}, volume 227 of {\em Graduate Texts in Mathematics}.
\newblock Springer-Verlag, New York, 2005.

\bibitem{gradedNO}
C.~N\u{a}st\u{a}sescu and F.~van Oystaeyen.
\newblock {\em Graded ring theory}, volume~28 of {\em North-Holland Mathematical Library}.
\newblock North-Holland Publishing Co., Amsterdam-New York, 1982.

\bibitem{stanleyTFR}
R.~Stanley.
\newblock Generalized {$H$}-vectors, intersection cohomology of toric varieties, and related results.
\newblock In {\em Commutative algebra and combinatorics ({K}yoto, 1985)}, volume~11 of {\em Adv. Stud. Pure Math.}, pages 187--213. North-Holland, Amsterdam, 1987.

\end{thebibliography}

\end{document}